\documentclass[12pt,draft,reqno]{amsart}

\usepackage{amsmath,amssymb,amscd,amsfonts,amsthm}

\usepackage{dsfont}
\usepackage[all]{xy}
\usepackage{enumerate}

\usepackage{paralist}

\usepackage{mathrsfs}

\usepackage{array,amscd,latexsym}
\usepackage{enumerate}
\usepackage[all]{xy}

{\catcode`\@=11
\gdef\n@te#1#2{\leavevmode\vadjust{%
 {\setbox\z@\hbox to\z@{\strut#1}%
  \setbox\z@\hbox{\raise\dp\strutbox\box\z@}\ht\z@=\z@\dp\z@=\z@%
  #2\box\z@}}}
\gdef\leftnote#1{\n@te{\hss#1\quad}{}}
\gdef\rightnote#1{\n@te{\quad\kern-\leftskip#1\hss}{\moveright\hsize}}
\gdef\?{\FN@\qumark}
\gdef\qumark{\ifx\next"\DN@"##1"{\leftnote{\rm##1}}\else
 \DN@{\leftnote{\rm??}}\fi{\rm??}\next@}}

\DeclareFontFamily{OT1}{wncyr}{\hyphenchar\font45 }
\DeclareFontShape{OT1}{wncyr}{m}{n}{%
   <5> <6> <7> <8> <9> gen * wncyr
   <10> <10.95> <12> <14.4> <17.28> <20.74>  <24.88>wncyr10}{}
\DeclareFontShape{OT1}{wncyr}{m}{it}{%
   <5> <6> <7> <8> <9> gen * wncyi
   <10> <10.95> <12> <14.4> <17.28> <20.74> <24.88> wncyi10}{}
\DeclareFontShape{OT1}{wncyr}{m}{sc}{%
   <5> <6> <7> <8> <9> <10> <10.95> <12> <14.4>
   <17.28> <20.74> <24.88>wncysc10}{}
\DeclareFontShape{OT1}{wncyr}{b}{n}{%
   <5> <6> <7> <8> <9> gen * wncyb
   <10> <10.95> <12> <14.4> <17.28> <20.74> <24.88>wncyb10}{}
\input cyracc.def

\DeclareMathSizes{9}{9}{7}{5}

\DeclareMathSymbol{\twoheadrightarrow} {\mathrel}{AMSa}{"10}

\theoremstyle{plain}

\newtheorem{theorem}{Theorem}
\newtheorem*{thmnonumber}{Theorem}
\newtheorem{proposition}{Proposition}
\newtheorem{lemma}{Lemma}
\newtheorem{corollary}{Corollary}
\theoremstyle{definition}

\newtheorem{example}{Example}
\newtheorem{remark}{\it Remark}

\theoremstyle{remark}

\def\Q{{\mathbb Q}}
\def\C{{\mathbb C}}
\def\Z{{\mathbb Z}}

\begin{document}

\title[Root lattices in number fields]
{Root lattices
in number fields}

\author[Vladimir L. Popov]{Vladimir L. Popov${}^{1}$}
\thanks{\today.
\newline \indent
${}^1$  Steklov Mathematical Institute,
Russian Academy of Sciences, Gub\-kina 8, Mos\-cow
119991, Russia, {\rm popovvl@mi-ras.ru}.
}

\author[Yuri G. Zarhin]{Yuri G. Zarhin${}^{2}$}
\thanks{${}^{2}
$ Department of Mathematics,
Pennsylvania State University, Uni\-versity Park, PA
16802, USA, {\rm zarhin@math.psu.edu}.\\
\indent The second named author (Y.Z.) was
partially supported by Simons Foundation
Collaboration grant $\# 585711$.
Part of this work was done during his stay at
the Weizmann Institute of Science
(December 2019--January 2020), whose
hospitality and support are gratefully acknowledged.
}

\begin{abstract}
We explore whether a root lattice may be similar to
the lattice
$\mathscr O$
of inte\-gers  of a
number
field
$K$
endowed with the inner product
$(x, y):={\rm Trace}_{K/\Q}(x\cdot\theta(y))$,
where $\theta$ is
an involution of $K$.
We classify all  pairs $K$, $\theta$ such that
$\mathscr O$ is similar
to either an even root lattice or
the root lattice $\mathbb Z^{[K:\Q]}$.
We also classify all  pairs
$K$, $\theta$ such that $\mathscr O$
is a root lattice.
In addition to this, we show that $\mathscr O$
is never similar to a
positive-definite
even unimodular lattice of rank $\leqslant 48$,
in particular, $\mathscr O$ is not similar to
the Leech lattice. In appendix,
we give a general cyclicity criterion for
the primary components of the discriminant
group of $\mathscr O$.
\end{abstract}

\maketitle

\section{Introduction}

Number fields  are natural sources of
{\it lat\-ti\-ces}, i.e., pairs $(L, b)$, where $L$
is a free $\mathbb Z$-module
of finite rank and $b\colon L\times L\to \mathbb Z$ is a
nondegenerate sym\-met\-ric bilinear form; see \cite{CP}.
Namely, let $K$ be a number field,
\begin{equation*}
n:=[K: \Q]<\infty,
\end{equation*}
and let ${\mathscr O}$ be  the ring
of integers of $K$.
We fix a field automorphism
\begin{equation}\label{inv}
\theta\in {\rm Aut}\hskip .3mm K\;\;
\mbox{such that $\theta$ is involutive, i.e., $\theta^2={\rm id}$}.
\end{equation}
Then the map
\begin{equation}\label{ti}
{\rm tr}_{K, \theta}\colon K\times K\to
\mathbb Q,\quad {\rm tr}_{K, \theta}(x, y):=
{\rm Trace}_{K/\mathbb Q}^{\ }(x\!\cdot\! \theta(y))
\end{equation}
is a nondegenerate symmetric bilinear form such that
for every nonzero ideal $I$ in ${\mathscr O}  $, the pair
$(I, {\rm tr}_{K, \theta}):=(I, {\rm tr}_{K, \theta}|_{I\times I})$
is a lattice of rank $n$.

This construction admits a natural generalization, see \cite{BF1}, \cite{BF4}, \cite[Chap. 8, \S7]{CS}, and refe\-ren\-ces therein.
Namely, let $J$ be a nonzero (frac\-tional) ideal of $K$, let $a\in K$ be
a nonzero element
such that $\theta(a)=a$,
${\rm Trace}_{K/\mathbb Q}(ax\cdot\theta(y))\in\mathbb Z$ for all $x, y\in J$, and\;let
\begin{equation*}
{\rm tr}_{K, \theta, J, a}\colon J\times J\to
\mathbb Z,\quad {\rm tr}_{K, \theta, J, a}(x, y):=
{\rm Trace}_{K/\mathbb Q}^{\ }(ax\!\cdot\! \theta(y)).
\end{equation*}
Then $(J, {\rm tr}_{K, \theta, J, a})$ is a lattice.\;The origins of this construction essentially go back to Gauss. Indeed,
for $n=2$ and $a=1/{\rm Norm}_{K/\mathbb Q}(J)$,
it turns into classical Gauss' construction, which yields correspondence bet\-ween ideals and binary quadratic forms.

Some remarkable lattices
are isometric to
the lattices
of the form
$(J, {\rm tr}_{K, \theta, J, a})$.
For instance, if $K$ is an $m$th cyclotomic field, then this is so for the root lattices
${\mathbb A}_{p-1}$ (with prime $p$),
${\mathbb E}_6$, and ${\mathbb E}_8$,
where,
respectively, $m=p$, $m=9$, and $m=15$, $20$, $24$. If $m=21$, this is so for the Coxeter--Todd lattice, and if $m=35$, $39$, $52$, $56$, $84$, for the Leach lattice. A classification of root lattices
isometric to the lattices of the form $(J, {\rm tr}_{K, \theta, J, a})$ for cyclotomic $K$ is given in \cite{BM}. References and more examp\-les see in
\cite{BF1}, \cite{BF4},
\cite[Chap. 8, \S7]{CS}. So,  given a lattice $(L, b)$, it arises
the problem of
finding out
whether it is isometric
to $(J, {\rm tr}_{K, \theta, J, a})$ for
suitable  $K$, $\theta$, $J$,\;$a$.

Another problem is to find out whether, for a given  lattice $(L, b)$ and a nonzero ideal $J$ of $K$,
there exist $\theta$ and $a$ such that $(J, {\rm tr}_{K, \theta, J, a})$ is isometric to $(L, b)$.

Among all nonzero ideals, there is a distinguished one, namely,
${\mathscr O}$
itself, for which $a=1$ is a distinguished value suitable for every $\theta$.
This leads to
the problem of finding
remarkable lattices
isometric (or, more generally, similar) to
lattices of
the form $({\mathscr O}  , {\rm tr}_{K, \theta})$.

The present paper is aimed to explore  whether $({\mathscr O}  ,
{\rm tr}_{K, \theta})$ may be similar to a
root lattice, in particular, whether  $({\mathscr O},
{\rm tr}_{K, \theta})$ itself may be a root lattice.
It naturally conjuncts with our previous
publication \cite{PZ}: both papers stem from our wish
to explore realizations in
number fields
of objects
associated with root systems.

There is
a classical construction of
geometric representation of al\-geb\-raic numbers, which embeds $K$  into a Eucli\-dean space, see, e.g., \cite[Chap.\;2, Sect.\;3]{BS}, \cite[6.1.2, 10.3.1]{TVN}.\;One can ask in which cases this em\-bedding
endows $\mathscr O$ with a structure of a lattice $(\mathscr O, b_K)$ isometric to a root lattice.
If the latter holds, then  necessarily $b_K(\mathscr O\times\mathscr O)\subseteq \mathbb Q$.
In Proposi\-tion\;\ref{comp} below is proved that this inclusion is equi\-valent to the existence of an involutive
automorphism $\theta\in {\rm Aut}\,K$ such that $b_K={\rm tr}_{K, \theta}$.\;Therefore,
the construc\-tion
of geometric repre\-sentation
does not provide a new
(in comparison with the one we ex\-p\-lore here)
possibility of naturally endowing $\mathscr O $ with a structure of lat\-tice isometric (or even similar)  to a root one.

\vskip 1.5mm

Before
formulating
our results, we recall some definitions
and facts
(see \cite[Chap. 4]{Ma}, \cite{Bo}, \cite{CS}, \cite{MH}),
and introduce some notation.

\vskip 1.5mm

\begin{asparaitem}

\item[$
\raisebox{.2\height}{\mbox{$\centerdot$}}$]
A nonzero lattice is called a {\it root lattice}
if it is isometric to ortho\-gonal direct sum of lattices
belonging to the
union of
two infinite series  ${\mathbb A}_\ell$\;$(\ell\geqslant 1)$,
${\mathbb D}_\ell$\;$(\ell\geqslant 4)$, and
four sporadic lattices $\mathbb Z^1$, ${\mathbb E}_6$,
${\mathbb E}_7$, ${\mathbb E}_8$,
who\-se explicit description is recalled in Appendix 1
below {\rm(}Section {\rm\ref{ap}}{\rm)}.
All lattices in this union
are indecomposable (i.e., inexpressible as orthogo\-nal
direct sums of nonzero summands).
By Eichler's theorem \cite[Thm. 6.4]{MH} decomposition
of a root lattice as orthogonal direct sum of
inde\-com\-po\-sable lattices (called its
{\it inde\-com\-po\-sable components}) is unique.

\item[$
\raisebox{.2\height}{\mbox{$\centerdot$}}$]
 Given a lattice $(L, b)$, we denote the orthogonal
 direct sum of $s>0$ copies of  $(L, b)$ by $(L, b)^s$.
 For $(L, b)=\mathbb Z^1$, we denote $(L, b)^s$ by $\mathbb Z^s$.

\item[$
\raisebox{.2\height}{\mbox{$\centerdot$}}$]
A characterization of root lattices
is given by fun\-da\-men\-tal Witt's theorem:

\begin{thmnonumber}[{{\rm E. Witt \cite{W};
see also\,\cite[Thm. 4.10.6]{Ma}}}]
A lattice $(L, b)$ is
a root lattice if and only if the following two conditions hold:
\begin{enumerate}[\hskip 2.2mm\rm(i)]
\item the form $b$  is positive-definite;
\item the $\mathbb Z$-module $L$ is generated by
the set $\{x\!\in\! L\mid  \mbox{$b(x, x)\!=\!1$ or $2$}\}$.
\end{enumerate}
\end{thmnonumber}

\item[$
\raisebox{.2\height}{\mbox{$\centerdot$}}$]
The lattices $(L_1, b_1)$ and $(L_2, b_2)$ are called {\it similar}
(equivalently,
one of them
is called {\it similar to}
the other)
if there are nonzero integers $m_1, m_2$ such that
the lattices
$(L_1, m_1b_1)$ and $(L_2, m_2b_2)$ are isometric.

\item[$
\raisebox{.2\height}{\mbox{$\centerdot$}}$]
A nonzero lattice $(L, b)$ is called a {\it primitive lattice}
if the positive integer
\begin{equation}\label{m}
d_{(L, b)}:={\rm gcd}\{b(x, y)\mid x, y\in L\},
\end{equation}
is $1$.
For every nonzero lattice $(L, b)$, the
lattice $(L, b/d_{(L, b)})$ is primitive.
Two  lattices $(L_1, b_1)$ and $(L_2, b_2)$
are similar if and only if
the lattices
$(L_1, b_1/d_{(L_1, b_1)})$ and $(L_2, b_2/d_{(L_2, b_2)})$
are isometric.
A root lat\-tice is non\-primitive if
and only  if
it is isometric to $\mathbb A_1^{a_1}$ for some $a_1$.
\end{asparaitem}

\vskip 1.5mm

We first consider a special case of the problem,
namely, explore whether $({\mathscr O}  ,
{\rm tr}_{K, \theta})$ may be a root lattice.
The following examples show that such cases
do exist.

\begin{example}\label{ex1}
Let $n=1$. Then  we have $K=\mathbb Q$,
${\mathscr O}  =\mathbb Z$, $\theta={\rm id}$, and
${\rm Trace}_{K/\mathbb Q}^{\ }(x) =x$ for every
$x\in K$. Hence in this case $({\mathscr O},
{\rm tr}_{K, \theta})$  is the root lattice $\mathbb Z^1$
(which is similar but not isometric to ${\mathbb A}_1$).
\end{example}

\begin{example}\label{ex2}
Let $n=2$ and let $K$ be a $3$rd cyclotomic
field: $K={\mathbb Q}(\sqrt{-3})$. Let $\theta$
be the complex conjuga\-ti\-on. Then
${\mathscr O}  ={\mathbb Z} + {\mathbb Z}\omega$,
where
$\omega=(1+\sqrt{-3})/2$,  and
\begin{equation}\label{fdt}
\mbox{${\rm Trace}_{K/\mathbb Q}^{\ }(x) \!=x+\theta(x)
=2{\rm Re}(x)$ for every $x\in K$}.
\end{equation}
This shows that $({\mathscr O},
{\rm tr}_{K, \theta})$ is a root lattice
isometric to
${\mathbb A}_2$; see \cite{PZ}.
\end{example}

\begin{example}\label{ex3}
Let $n=2$ and let $K$ be a $4$th cyclotomic field:
$K={\mathbb Q}(\sqrt{-1})$. Let $\theta$ be
the complex conjuga\-ti\-on. Then
${\mathscr O}  ={\mathbb Z} + {\mathbb Z}\sqrt{-1}$,
and formula \eqref{fdt} still holds. This shows that $({\mathscr O},
{\rm tr}_{K, \theta})$ is a root lattice isometric to
$\mathbb A_1^{2}$; see \cite{PZ}.
\end{example}

Our first main result,
Theorem \ref{ridnew} below,
yields the classification of all  pairs $K$, $\theta$
for which $({\mathscr O} ,
{\rm tr}_{K, \theta})$ is a root lattice:

\begin{theorem}
\label{ridnew}
The following properties of a pair $K, \theta$ are equivalent:
\begin{enumerate}[\hskip 4.2mm \rm(a)]
\item $(\mathscr O, {\rm tr}_{K, \theta})$ is a root lattice;
\item $K$, $\theta$ is one of the following pairs:
\begin{enumerate}[\hskip 0mm \rm(i)]
\item[$({\rm b}_1)$] $K=\mathbb Q$, $\theta={\rm id};$
\item[$({\rm b}_2)$] $K=\mathbb Q(\sqrt{-3})$, $\theta$
is the complex conjugation;
\item[$({\rm b}_3)$] $K=\mathbb Q(\sqrt{-1})$, $\theta$
is the complex conjugation.
\end{enumerate}
\end{enumerate}
\end{theorem}

We then
address
the general problem of classifying all pairs $K$,
$\theta$ such that
the lattice
$(\mathscr O, {\rm tr}_{K, \theta})$
is
similar (but not necessarily isometric) to a root
lattice $(L, b)$. It appears that
such pairs are far from being exhausted by
Examples  \ref{ex1}, \ref{ex2}, and \ref{ex3}.
We obtain their complete classifica\-tions
in both
``unmixed'' cases, namely, when the $\Z$-module
$L$ is gene\-rated
by the set $\{x\in L\mid b(x, x)=1\}$
and when it generated by the set
$\{x\in L\mid b(x, x)=2\}$.
The first case
is precisely the one
in which
$(L, b)$ is isometric to
$\mathbb Z^n$.
The second
is the one
in which every
indecomposable component of
$(L, b)$ is not isometric to
$\mathbb Z^1$; the latter property, in turn,
is equivalent to the evenness of the lattice $(L, b)$.
 Our next two main results,  Theorem \ref{pureodd} and  \ref{simeven}
 below, yield
  these classifications.
 In these theorems, $m$ denotes the unique positive integer
 such that
 \begin{equation}\label{mmmm}
 {\rm Trace}_{K/\Q}(\mathscr O)=m\mathbb Z
 \end{equation}
 (such $m$ exists because ${\rm Trace}_{K/\Q}\colon
 \mathscr O\to \mathbb Z$ is a nonzero additive
 group homomorphism).

\begin{theorem}
\label{pureodd}
The following properties of a pair $K, \theta$ are equivalent:
\begin{enumerate}[\hskip 4.2mm \rm(a)]
\item\label{pureodda} $(\mathscr O, {\rm tr}_{K, \theta})$
is similar to
$\mathbb Z^n$;
\item\label{pureoddb} $(\mathscr O, {\rm tr}_{K, \theta})$
is similar to
$\mathbb A_1^{n}$;
\item\label{pureoddc} $(\mathscr O, {\rm tr}_{K, \theta}/m)$
is isometric to
$\mathbb Z^{n}$;
\item\label{pureoddd} $(\mathscr O, 2\hskip .4mm{\rm tr}_{K, \theta}/m)$
is isometric to
$\mathbb A_1^{n}$;
\item\label{pureodde}
$K$ is a $2^a$th cyclotomic field
for a positive integer $a$,
and $\theta$
is the complex conjugation if $a>1$, and $\theta={\rm id}$ if $a=1$.
\end{enumerate}

If  these properties
hold, then $n=2^{a-1}$ and $m=n$.

In \eqref{pureodde}, let $\zeta_{2^a}\!\in\! K$ be
a $2^a$th primitive root of unity, and let
$x_j\!:=\!\zeta_{2^a}^j$. Then the set of
all  indecomposable components of the root
lattice $(\mathscr O, {\rm tr}_{K, \theta}/m)$
coincides with the set of all its
sublattices
\begin{equation*}
\mathbb Zx_j,\quad 0\leqslant j\leqslant 2^{a-1}-1.
\end{equation*}
For every $j$, the value of ${\rm tr}_{K, \theta}/m$ at
$(x_j, x_j)
$ is $1$.
\end{theorem}

\begin{theorem}
\label{simeven}
The following
properties of a pair $K, \theta$  are equivalent:
\begin{enumerate}[\hskip 4.2mm \rm(a)]
\item\label{simeva} $(\mathscr O, {\rm tr}_{K, \theta})$
is similar to an even primitive root lattice.
\item\label{simevb} $(\mathscr O, {\rm tr}_{K, \theta}/m)$
is an even primitive root lattice.
\item\label{simevc} $n$ is even and
$(\mathscr O, {\rm tr}_{K, \theta}/m)$ is isometric to
$\mathbb A_2^{n/2}$.
\item\label{simevd}
$K$ is a $2^a3^b$th cyclotomic field for some
positive integers $a$ and $b$,
  and $\theta$ is the complex conjugation.
    \end{enumerate}

   If these properties hold, then $n=2^{a}3^{b-1}$ and $m=n/2$.

     In \eqref{simevd}, let $\zeta_{2^a}$ and $\zeta_{3^b}\in K$
     be respectively a primitive
     $2^a$th and $3^b$th root of unity, and let
     $x_{i,j}:=\zeta_{2^a}^i\zeta_{3^b}^j$.
     Then the set of all  indecomposable components of
     the root lattice $(\mathscr O, {\rm tr}_{K, \theta}/m)$
     coincides with the set of all
     its sublattices
     \begin{equation*}
     \mathbb Zx_{i, j}+\mathbb Z x_{i, j+3^{b-1}} \quad
     0\leqslant i\leqslant 2^{a-1}-1,\;\;0\leqslant j\leqslant 3^{b-1}-1.
     \end{equation*}
     For all $i, j$, the values of ${\rm tr}_{K, \theta}/m$ at
   $(x_{i,j}, x_{i,j})$,
    $(x_{i,j+3^{b-1}}, x_{i,j+3^{b-1}})$, and $(x_{i,j}, x_{i,j+3^{b-1}})$
   are
   respectively $2$, $2$, and $-1$.
\end{theorem}

Note that if $K$ is a $d$th cyclotomic field, and $\theta$ is the complex con\-jugation,
then for $d=2^a$, the similarity of $(\mathscr O, {\rm tr}_{K, \theta})$ to $\mathbb Z^{2^{a-1}}$
was observed in \cite[Prop.\;9.1(ii)]{BF2}, and for $d=2\cdot 3^b$, the similarity of
$(\mathscr O, {\rm tr}_{K, \theta})$ to $\mathbb A_2^{3^{b-1}}$
can be deduced from \cite[Prop.\;9.1(i)]{BF2}.

Since $\mathbb E_8$ is the unique (up to isometry)
positive-definite even uni\-mo\-du\-lar lattice
of rank $8$ (see \cite[\S6]{MH}),
Theorem \ref{simeven} solves in the negative  for $n=8$
the existence problem
of a lattice $({\mathscr O}, {\rm tr}_{K, \theta})$
similar to a positive-definite even unimodular one.
Our last main result, Theorem\;\ref{24} below,
shows that as a matter of fact
 the following more general state\-ment holds:

\begin{theorem}
\label{24}
Every positive-definite even unimodular lattice
of rank $\leqslant 48$
is not similar to a lattice of the form
$(\mathscr O, {\rm tr}_{K, \theta})$.
\end{theorem}

\begin{corollary}
Every lattice $(\!\mathscr O, {\rm tr}_{K, \theta})$
is not similar to the Leech lattice.
\end{corollary}

Theorem \ref{24} excludes
many lattices from being similar to
lattices
of the form $(\!\mathscr O, {\rm tr}_{K, \theta})$:
recall \cite[Chap. 2, \S6]{MH} that if $\Phi(r)$
is the number of pairwise nonisometric
positive-definite even unimodular lattices
of rank $r$, then $\Phi(8)=1$,
$\Phi(16)=2$, $\Phi(24)=24$, $\Phi(32)\geqslant
10^7$, $\Phi(48)\geqslant 10^{51}$.

\vskip 1.5mm

This paper
is organized
as follows.
Theorems \ref{ridnew}, \ref{pureodd}, \ref{simeven},
and \ref{24} are pro\-ved, respectively, in Sections
\ref{rola}, \ref{szn}, \ref{prprl}, and \ref{prpul}.
Section \ref{ge} contains se\-ve\-ral general auxiliary
results
used in these  proofs
and in the proof of Theorem \ref{cdg} (see below).  At the end of this paper,
two short appendices are placed.
Appendix 1 (Section \ref{ap}) recalls the explicit
description of indecomposable root lattices.
Appendix 2 (Section \ref{ap2})
contains a general cyclicity criterion for the
primary components of the discrimi\-nant group
of $(\mathscr O, {\rm tr}_{K, \theta})$
(Theorem \ref{cdg}).
In the first stage of this project,
we
used another approach to finding a classification of
all  pairs $K$, $\theta$ such that
$(\mathscr O, {\rm tr}_{K, \theta})$
is
similar to
an indecomposable even root lattice.
This approach
led us
to only a partial
answer,
see \cite[Thms. 4, 5]{P} (in contrast, in
the present paper we obtain
a complete classification, see Theorems
\ref{pureodd} and \ref{simeven}).
However, within this approach we obtained
and applied a
general cyclicity crite\-rion, which may be
useful for other applications. Therefore, we
consider it worthwhile to publish
it in Appendix 2 as Theorem \ref{cdg}.

\vskip 2mm

{\it Acknowledgements.} We thank E.\;Bayer-Fluckiger for her interest in this work
and help with references, and the referee,  whose
wish for the authors that they
compare
the two lattice constructions resulted in proving
Pro\-position \ref{comp}.

\vskip 2mm
{\it Conventions, terminology, and notation}

\vskip 2mm

Given a lattice $(L, b)$ of rank $r>0$, we canonically
embed $L$ in the vector space $V=
L\otimes _{\mathbb Z}\mathbb Q$ over $\mathbb Q$
and extend $b$ to a nondegenerate symmetric
bilinear form $V\times V\to \mathbb Q$,
still denoted by $b$. If $(L, b)=({\mathscr O}, {\rm tr}_{K, \theta})$,
then $r=n$ and $V$ is naturally identified with $K$.
If $M$ is a submodule of the $\mathbb Z$-module $L$,
then $(M, b|_{M\times M})$ is called a {\it sublattice}
of $(L, b)$ and denoted just by $M$.

For every
$s\in \mathbb Z$, we put
\begin{equation}\label{ccc}
(L, b)[s]:=\{x\in L\mid b(x, x)=s\}.
\end{equation}
Thus $(L, b)$ is an {\it even lattice} if and only if $(L, b)[s]=
\varnothing$ for all odd $s$.

${\rm discr}(L, b)$
is the {\it discriminant} of $(L, b)$, i.e.,
\begin{equation}\label{discr}
{\rm discr}(L, b):=\det (b(e_i, e_j))\in\Z,
\end{equation}
where $e_1,\ldots, e_r$ is a basis of $L$ over
$\mathbb Z$ (the right-hand side of \eqref{discr}
does not depend on the choice of basis).

$L^*$ is the {\it dual of $L$ with respect to $b$}, i.e.,
\begin{equation}\label{duanew}
L^*:=\{x\in V\mid \mbox{$b(x, L)\subseteq
\mathbb Z$}\}
\supseteq L.
\end{equation}
The {\it discriminant group}
of $(L, b)$
is the (finite Abelian) group $L^*/L$.

${\rm discr}\,{K/\mathbb Q}:=
{\rm discr}({\mathscr O}, {\rm tr}_{K, \rm id})$ is
the discriminant of
$K/\mathbb Q$.

${\rm Trace}_{K/\mathbb Q}^{\ } (x)$ and
${\rm Norm}_{K/\mathbb Q}^{\ } (x)$ are
respectively the trace and norm over $\mathbb Q$
of an element $x\in K$.

$\mathfrak c$ and $\mathfrak d$ are respectively
the {\it codifferent} and {\it different} of
$K/\mathbb Q$, i.e.,
\begin{equation}\label{cd}
\begin{gathered}
\mbox{$\mathfrak c$ is the dual of ${\mathscr O}$
with respect to ${\rm tr}_{K, \rm id}$},\\[-1.5mm]
\mathfrak d
:={\mathfrak c}^{-1}
:=
\{x\in K \mid x {\mathfrak c}
\subseteq \mathscr O\}
\end{gathered}
\end{equation}
 ($\mathfrak c$ is a fractional ideal
 of $K$,
 ${\mathscr O}\subset \mathfrak c$, and
 $\mathfrak d$ is an ideal
 of ${\mathscr O}$) \cite{Koch}, \cite{Neukirch}.

 If $\mathfrak a$ and $\mathfrak b$ are nonzero
 ideals of $\mathscr O$ and $\mathfrak a$ is prime,
 then ${\rm ord}_{\mathfrak a}{\mathfrak b}$ is
 the highest nonnegative integer $t$ such that
${\mathfrak a}^t\supseteq \mathfrak b$.

$
K^\theta:=\{x\in K\mid \theta(x)=x\}.
$

$\mu_K$ is the (finite cyclic) multiplicative
group of all roots of unity in\;$K$.

$\varphi$ is Euler's totient function.

$\overline z$ stands for the complex conjugate of $z\in\mathbb C$.

 Given a field $F$, its multiplicative group is denoted by $F^\times$.

\section{Generalities}\label{ge}

In this section, several auxiliary results are collected.
We do not claim priority for all of them and give a
reference always when we know it. At the same time, we
prove all the statements, wanting to make this paper
reasonably  self-contained, and because our proofs,
while being rather short, provide somewhat more information
than can be found in the literature.

 \begin{lemma}
 \label{pthpower}
The following properties of a prime $p$ are equivalent:
\begin{enumerate}[\hskip 2.8mm\rm(i)]
 \item
 ${\rm Trace}_{K/\mathbb Q}^{\ }(x)\in
 p {\mathbb Z}\;\;\mbox{for all  $x\in {\mathscr O}$}$.
 \item
 ${\rm Trace}_{K/\mathbb Q}^{\ }(x^p)\in p{\mathbb Z}\;\;
 \mbox{for all  $x\in {\mathscr O}  $}$.
 \end{enumerate}
  \end{lemma}

 \begin{proof} (i)$\Rightarrow$(ii) is clear. We shall prove
 the inclusion
 \begin{equation}\label{cong}
 {\rm Trace}_{K/\mathbb Q}^{\ }(x^p)- ({\rm Trace}_{K/\mathbb Q} (x))^p
 \in p {\mathbb Z}
 \;\; \mbox{for every $x\in {\mathscr O}$},
 \end{equation}
 which readily implies (ii)$\Rightarrow$(i).
 The set of all field embeddings
  $K \hookrightarrow {\mathbb C}$ contains exactly
  $n$ elements $\sigma_1,\ldots, \sigma_n$, and for
  every $x\in K$, we have
\begin{equation}\label{sum}
{\rm Trace}_{K/\mathbb Q}^{\ } (x)=
x_1+\cdots +x_n,\;\;\mbox{where $x_i:=\sigma_i(x)$}
\end{equation}
(see, e.g., \cite[Chap.\;12]{IR}).
From \eqref{sum} we deduce the existence of a
sym\-met\-ric polynomial $f=f(t_1,\ldots, t_n)$ in
variables $t_1,\ldots, t_n$ with integer coefficients such that
   \begin{equation}\label{xpx}
   \begin{split}
   {\rm Trace}_{K/\mathbb Q}^{\ }(x^p)\!-\!
   ({\rm Trace}_{K/\mathbb Q} (x))^p  &=
   (x_1^p\!+\!\cdots\!+\!x_n^p)\! -\!
   (
   x_1\!+\!\cdots \!+\! x_n
   )^p\\
   & =
  p \cdot f(x_1, \dots, x_n).
  \end{split}
  \end{equation}

  Let  $s_i=s_i(t_1,  \dots, t_n)$ be the elementary
  symmetric polynomial
  in $t_1,\ldots, t_n$ of degree $i$. Then
  $f$ may be represented a polynomial with integer
  coefficients in $s_1, \ldots, s_n$
   (see, e.g., \cite[Chap.\,V, Thm. 11]{Lang}).
   Since, for $x\in \mathscr O$,   we have $s_i(x_1,
   \dots, x_n)\in \mathbb Z$ for all $i$, this implies
   that $f(x_1,  \dots, x_n)\in \mathbb Z$. The latter
   inclusion and  \eqref{xpx} yield \eqref{cong}.
 \end{proof}

\begin{corollary}
\label{evenT}
The lattice $({\mathscr O}  , {\rm tr}_{K, \rm id})$ is even if
and only if the integer
$m$
is even.
\end{corollary}

\begin{proof}
This follows from \eqref{mmmm} and
Lemma \ref{pthpower} with $p=2$.
\end{proof}

\begin{lemma}\label{prp} Let $(L, b)$ be a
nonzero lattice of rank $r$.

\begin{enumerate}[\hskip 2.2mm \rm(i)]
\item\label{sss} If $e_1,\ldots, e_r$ is a basis
of the $\Z$-module $L$
and
$s_1,\ldots, s_r$ are the in\-va\-riant factors  of
the matrix
$(\hskip -.2mm b(e_i,\hskip -.3mm e_j)\!)$ {\rm(}see
{\rm \cite[(16.6)]{CR}}{\rm)}, then the group
$$\textstyle\bigoplus_{i=1}^r \mathbb Z/s_i{\mathbb Z}$$
is isomorphic to the discriminant group of $(L, b)$.
    In par\-ti\-cu\-lar, $
s_1\cdots s_r=|{\rm discr}(L, b)|$ is the order of the latter group.

\item\label{test}
${\rm discr}(L, b)=d_{(L, b)}^{\,r} {\rm discr}(L, b/d_{(L, b)})$.

\item\label{m_d} If $(L, b)\!=\!({\mathscr O},
 {{\rm tr}}_{\theta})$, then
 \begin{enumerate}[\hskip 0mm \rm(i)]
 \item[$\rm (\ref{m_d}_1)$]\label{m_d1}
 $d_{(L, b)}\!=\!m$ {\rm(}in particular, $d_{(L, b)}$
 is independ of $\theta${\rm)};
\item[$\rm (\ref{m_d}_2)$]\label{ttilde}
$n\equiv 0\;{\rm mod}\,m$.
\end{enumerate}
\item\label{dgdg}
The discriminant group of
$(\mathscr O, {\rm tr}_{\theta})$ is iso\-mor\-phic
to\;${\mathscr O}/\mathfrak d$
{\rm(}in par\-ti\-cular, up to isomorphism,
 this group is independent of
 $\theta${\rm)}. Its order is equal to
 $|{\rm disc}\,K/\mathbb Q|$ {\rm(}cf.\;{\rm\cite[Cor.\;2.4]{BF2})}.

\item\label{te} Let
$(L, b/d_{(L, b)})$ be an even lattice.\;Then
$(L,  b)$ is also an even lattice, and if, moreover,
$(L, b)=({\mathscr O}, {{\rm tr}}_{\theta})$,
then
    \begin{enumerate}[\hskip 0mm \rm(a)]
    \item[${\rm (\ref{te}_1)}$]\label{81}
    $n\equiv 0\;{\rm mod}\, 2$ and $n/2\equiv 0\;{\rm mod}\,m$;
\item[${\rm (\ref{te}_2)}$]\label{82}
$n\equiv 0\;{\rm mod}\, 4$ for $\theta={\rm id}$.
\end{enumerate}
\end{enumerate}
\end{lemma}

\begin{proof}
\eqref{sss} Let $e^{*}_1,\ldots, e_{r}^{*}$ be
the basis of $V$ dual to $e_1,\ldots, e_r$ with respect to $b$, i.e.,
$b(e_i, e_j^{*})=\delta_{ij}$. In view of \eqref{duanew},
since $L=\mathbb Ze_1\oplus\cdots\oplus \mathbb Ze_n$,
we have $L^*=\mathbb Ze_1^*\oplus\cdots\oplus \mathbb Ze_n^*$.
On the other hand,
$e_i\!=\!\sum_{j=1}^r b(e_i, e_j)e_j^*$, i.e.,
$(\hskip -.2mm b(e_i,\hskip -.3mm e_j)\!)$ is  the
chan\-ge-of-basis matrix for
passing from $e_1^*,\ldots , e_n^*$
to $e_1,\ldots , e_n$. Since
 the Smith normal form of this matrix
 is ${\rm diag}(s_1,\ldots, s_r\!)$, this implies the claim.

\eqref{test} This follows from \eqref{m} and \eqref{discr}.

$\rm (\ref{m_d}_1)$ Since $\theta(\mathscr O)=\mathscr O$,
and $1\in\mathscr O$,
we have $\{x\cdot \theta(y)\mid x, y\in \mathscr O\}=
\mathscr O$. Combining this with
\eqref{ti}, \eqref{m}, \eqref{mmmm}, we obtain
the desired equality.

$\rm (\ref{m_d}_2)$ This follows from $\rm (\ref{m_d}_1)$ and \eqref{m} because
\begin{equation}\label{trn}
{\rm tr}_{K, \theta}(1, 1)={\rm Trace}_{K/\mathcal Q}(1)=n.
\end{equation}

\eqref{dgdg}
It follows from $\theta(\mathscr O)=\mathscr O$
and formulas
\eqref{ti}, \eqref{duanew}, \eqref{cd}
that $\mathfrak c$ is the dual of
${\mathscr O}$
with respect to ${{\rm tr}_{K, \theta}}$.
Therefore, the discriminant group of
$(\mathscr O, {\rm tr}_{K, \theta})$ is
$\mathfrak c/\mathscr O$. By \cite[(4.15), p.\;85]{CR},
the latter group
is isomorphic to $\mathfrak c\mathfrak d/\mathscr O\mathfrak d$.
Since $\mathfrak c\mathfrak d=\mathscr O$ and
$\mathscr O\mathfrak d=\mathfrak d$, this proves
the first claim.

By \eqref{sss},  the order of the discriminant group of
$(\mathscr O, {\rm tr}_{K, {\rm id}})$ is $|{\rm disc}\,K/\Q|$.
Since, by the first claim, this group is isomorphic to the
discriminant group of $(\mathscr O, {\rm tr}_{K, \theta})$,
this proves the second claim.

\eqref{te} The first claim follows from \eqref{m}.
Let $(L, b)=(\mathscr O, {\rm tr}_\theta)$.
By $\rm (\ref{m_d}_1)$, the lattice
$(\mathscr O, {\rm tr}_\theta/m)$ is even.
By \eqref{trn} this implies that $n/m$ is an even
integer; whence ${\rm (\ref{te}_1)}$.
If, moreover, $\theta={\rm id}$, then
$m$ is even by Corollary \ref{evenT}; whence
${\rm (\ref{te}_2)}$ because
of the second formula
in ${\rm (\ref{te}_1)}$.
\end{proof}

\begin{remark} By Corollary \ref{evenT} and
Lemma \ref{prp}$\rm (\ref{m_d}_1)$, if
an even lattice $(L, b)$ is isometric to a lattice of the form
$({\mathscr O}  , {\rm tr}_{K, \rm id})$,
then necessarily $d_{(L, b)}$ is even; in
particular, $(L, b)$ is not primitive.
For instance, this shows that $({\mathscr O} , {\rm tr}_{K, \rm id})$
cannot be a primitive even root lattice.
\end{remark}

For any root lattice $(L, b)$, the form $b$ is positive-definite.\;Hence
if $(L, b)$ and $({\mathscr O}, {\rm tr}_{K, \theta})$
are similar, then ${\rm tr}_{K, \theta}$
is  a definite form.
The fol\-lowing Lemma \ref{TRCM} describes when the latter happens.

\begin{lemma}\label{TRCM}
The following properties of a pair $K, \theta$
are equivalent:
\begin{enumerate}[\hskip 3.0mm\rm(i)]
\item  ${\rm tr}_{K, \theta}$
is a definite bilinear form;
\item ${\rm tr}_{K, \theta}$
is a positive-definite bilinear form;
\item\label{co} either $K$ is a totally real field and
$\theta={\rm id}$, or $K$ is a CM-field and $\theta$
is the complex conjugation.
\end{enumerate}
\end{lemma}
\begin{proof}
By \eqref{trn}, ${\rm tr}_{K, \theta}(1,1)>0$;
whence (i)$\Leftrightarrow$(ii).

Let $r_1$ (respectively $r_2$) be the number of real (respectively
pairs of imaginary) field embeddings
$K^\theta\hookrightarrow \mathbb C$. If
$\theta\neq {\rm id}$, then $[K: K^\theta]=2$ and
$K=K^\theta(\sqrt{a})$ for
$a\in K^\theta$; let $s$ be the number of field
embeddings $\iota\colon K^\theta\hookrightarrow
\mathbb R$ such that $\iota(a)<0$. The signature of
$({\mathscr O}, {\rm tr}_{K, \theta})$ is then given by
\begin{equation}\label{sign}
{\rm sign}({\mathscr O}, {\rm tr}_{K, \theta})=
\begin{cases}(r_1+r_2, r_2)&  \mbox{if $\theta={\rm id}$},\\
(r_1+2r_2+s, r_1+2r_2-s)&  \mbox{if $\theta\neq {\rm id}$};
\end{cases}
\end{equation}
 see \cite[Prop.\;2.2]{BF1}. From \eqref{sign} one
 readily deduces
 (ii)$\Leftrightarrow$(iii).
\end{proof}

Note that
\begin{equation}\label{pCM}
\begin{split}
&\mbox{if $K$ is a CM-field and $\theta$ is the complex}\\[-1mm]
&\mbox{conjugation, then $n$ is even, $[K^\theta:\mathbb Q]=n/2$,}\\[-1mm]
&\mbox{and $({\rm discr}\,{K^\theta/\mathbb Q})^2$
divides ${\rm discr}\,{K/\mathbb Q}$;}
\end{split}
\end{equation}
see, e.g., \cite[Chap.\;III, Cor.\;(2.10)]{Neukirch}.

\begin{corollary}\label{fie}
For a pair $K,\theta$, if the lattice $({\mathscr O}, {\rm tr}_{K, \theta})$
is similar to a root lattice, then either $K$ is a totally real field
and $\theta={\rm id}$, or $K$ is a CM-field and $\theta$
is the complex conjugation.
\end{corollary}

\begin{lemma}[{{\rm cf.\;\cite[Lem.\;2]{BF0}}}]
\label{CB}
Let $\sigma_1,\ldots, \sigma_n$ be the set of all field
embed\-dings $K \hookrightarrow \C$. Let
$x \in {\mathscr O}$ be an element such that $\sigma_i(x)$
is a positive real number for every\;$i$. Then the
following hold:
\begin{enumerate}[\hskip 2.2mm\rm(i)]
\item\label{1CB} ${\rm Trace}_{K/\mathbb Q}(x)$
is a positive integer and
\begin{equation}\label{intn}
{\rm Trace}_{K/\mathbb Q}(x)\geqslant n;
\end{equation}
 \item\label{2CB} the equality in {\rm \eqref{intn}}
 holds if and only if $x=1$.
 \end{enumerate}
\end{lemma}

\begin{proof} The equalities
\begin{equation}\label{sp}
\textstyle {\rm Trace}_{K/\mathbb Q}(x)=\sum_{i=1}^n\sigma_i(x),\quad {\rm Norm}_{K/\mathbb Q}(x)=\prod_{i=1}^n\sigma_i(x)
\end{equation}
 (see, e.g., \cite[Chap.\,12]{IR}) imply that
 $\textstyle {\rm Trace}_{K/\mathbb Q}(x)$ and
 ${\rm Norm}_{K/\mathbb Q}(x)$ are positive integers.
 Combining \eqref{sp} with
  the classic inequality of arith\-me\-tic and geometric
  means (see, e.g., \cite[Sect.\,2]{Ste}), we then obtain
  the in\-equ\-alities
  \begin{equation}\label{CBtn}
  \textstyle {\rm Trace}_{K/\mathbb Q}(x)\geqslant n \cdot{\rm Norm}_{K/\mathbb Q}(x)\geqslant n \cdot 1=n.
  \end{equation}
This proves (i).

 If $x=1$, then the equality in \eqref{intn} holds by \eqref{trn}.
 Conversely, assume that the equality in \eqref{intn}
 holds for some $x$.
 In view of \eqref{CBtn}
 we then con\-c\-lude that
 \begin{equation}\label{n=1}
 {\rm Norm}_{K/\mathbb Q}(x)=1
 \end{equation}
 and the equalities in \eqref{CBtn} hold too.
 The classic inequality of arith\-me\-tic and geometric
 means tells us that the latter happens
 if and only if
 $\sigma_1(x)=\ldots=\sigma_n(x)$, i.e., $x$ is
 a positive integer. From \eqref{n=1} and \eqref{sp} we
 then obtain $x^n=1$, hence $x=1$. This proves \eqref{2CB}.
\end{proof}

The inequality in Lemma \ref{tnin}(ii) below can be found in \cite[Cor.\;4]{BF4}.

\begin{lemma}\label{tnin}
Suppose a pair $K, \theta$ shares one of the following properies:
\begin{enumerate}[\hskip 4.2mm \rm(TR)]
\item[\rm(TR)]
$K$ is a totally real field, $\theta={\rm id}$;
\item[\rm(CM)] $K$ is a CM-field, $\theta$ is the complex conjugation.
\end{enumerate}
Then, for every nonzero element
$x\in \mathscr O$,
the following hold:
\begin{enumerate}[\hskip 2.2mm \rm(i)]
\item\label{1}
${\rm tr}_{K, \theta}^{\ } (x, x)$
is a positive
integer;
\item\label{3}
${\rm tr}_{K, \theta}^{\ } (x, x)\geqslant
n$ and the equality holds if and only if $x$ is a root of unity.
\end{enumerate}
The condition $(\mathscr O, {{\rm tr}_{K, \theta}}/m)[s]\neq \varnothing$
{\rm(}see {\rm \eqref{ccc})} implies the following:
\begin{enumerate}[\hskip 4.7mm\rm(a)]
\item\label{atnin} $s\geqslant n/m$;
\item\label{btnin} if $s=1$, then $m=n$;
\item\label{ctnin} if $s=2$ and $(\mathscr O, {{\rm tr}_{K, \theta}}/m)$ is
even,
then
$m=n/2$,
\item\label{dtnin} if $s=4$ and $(\mathscr O, {{\rm tr}_{K, \theta}}/m)$ is
even,
then
$m=n/2$ or $n/4$.
\end{enumerate}
\end{lemma}
\begin{proof}
First, we note that if $\sigma\colon K \hookrightarrow \mathbb C$
is a field embedding, then
$\sigma(x\cdot\theta(x))$ is a positive real number. Indeed, if (TR) holds,
then
\begin{equation}\label{reaa}
\sigma(K)\subset \mathbb R\;\;
\mbox{ and }\;\;\sigma(x\cdot\theta(x))=\sigma(x^2)=(\sigma(x))^2;
\end{equation}
whence the claim in this case.
If (CM) holds,
then $\sigma(K)$ is stable with respect to the complex
conjugation of $\mathbb C$ and $\sigma(\theta(x))=\overline{\sigma(x)}$
(see, e.g., \cite[p.\;38]{Wa}). Therefore,
\begin{equation}\label{conju}
\sigma(x\cdot\theta(x))=\sigma(x)\sigma(\theta(x))=
\sigma(x)\overline{\sigma(x)}=|\sigma(x)|^2,
\end{equation}
which proves the claim in this case.

In view of this, \eqref{1} and the inequality in \eqref{3}
follow from \eqref{ti} and Lemma  \ref{CB}\eqref{1CB}.
By Lemma  \ref{CB}\eqref{2CB}, the equality in \eqref{3}
holds if and only if
\begin{equation}\label{=1}
x\cdot \theta(x)=1.
\end{equation}
Assume that \eqref{=1} holds. Then, in the notation of Lemma \ref{CB},
every complex number $\sigma_i(x)$
has the modulus
1. Since $x$ is an algebraic integer,
by Kroneker's theorem \cite{Kron} (see also, e.g., \cite[Lem.\,1.6]{Wa})
this implies that $x$ is a root of unity. Conversely, if $x$
is a root of unity, then, in the above notation, $\sigma(x)$
is a root of unity too, hence the right-hand sides of the equalities
in \eqref{reaa}, \eqref{conju} are equal to 1. Whence, \eqref{=1} holds.
This completes the proof of \eqref{3}.

Suppose that $(\mathscr O,
{{\rm tr}_{K, \theta}}/m)[s]\neq \varnothing$, i.e., there is
$a\in \mathscr O$ such that $
{{\rm tr}_{K, \theta}}(a,a)/m=s$. This
and
\eqref{3} yield
\eqref{atnin}.

Since,  by Lemma \ref{prp}$\rm (\ref{m_d}_2)$,
$n/m$ is an integer, it follows from
\eqref{atnin} that $m=n$ if $s=1$. This proves \eqref{btnin}.

Suppose, moreover, that $(\mathscr O, {{\rm tr}_{K, \theta}}/m)$ is
even.
Then $n/m$ is even
by Lemma \ref{prp}${\rm (\ref{te}_1)}$.
This and \eqref{atnin} then imply that $m=n/2$ if $s=2$.
This proves \eqref{ctnin}.
Finally, if $s=4$, then \eqref{atnin} implies that $n/m=2$ or $4$.
This proves \eqref{dtnin}.
\end{proof}

\begin{lemma}[{{\rm cf.\;\cite[Prop.\;9.1(ii)]{BF2}}}]
\label{2acycl} Fix a positive integer $a$.
Let $K$ be a $2^a$th cyclotomic field,
let $\theta$ be the complex conjugation
if $a>1$ and $\theta={\rm id}$ if $a=1$, and let $\zeta\in K$
be a $2^a$th primitive root of unity. Then the following hold:
\begin{enumerate}[\hskip 4.2mm\rm(i)]
\item\label{2aci} $n=2^{a-1};$
\item\label{2acii} $\{\zeta^j\mid 0\leqslant j\leqslant 2^{a-1}-1\}$ is a basis of the $\Z$-module $\mathscr O;$
\item\label{2aciii} for every elements $\zeta^i$ and $\zeta^j$ of this basis,
$${\rm tr}_{K, \theta}(\zeta^i, \zeta^j)=\begin{cases}n\quad\mbox{if $i=j,$}\\
0\quad \mbox{if $i\neq j;$}\end{cases}$$
\item\label{2aciv} $m=n$.
\end{enumerate}
\end{lemma}
\begin{proof}
We have $n=\varphi(2^a)$ (see \cite[Chap.\,VIII, \S3, Thm.\,6]{Lang});
whence\;\eqref{2aci}.

\eqref{2acii} We have $K=\Q(\zeta)$; therefore, $\mathscr O=
\mathbb Z[\zeta]$ (see \cite[Thm. 2.6]{Wa}).
The latter yields \eqref{2acii} because
the degree of the minimal poly\-nomial of $\zeta$ over $\Q$
is $n={\rm rank}\,\mathscr O$.

\eqref{2aciii}  Since $\theta$ is the complex conjugation and
$\zeta$ is a root of unity, we have $\zeta^i\cdot\theta(\zeta^j)=
\zeta^{i-j}$, hence by \eqref{ti},
\begin{equation}\label{trac}
{\rm tr}_{K, \theta}(\zeta^i, \zeta^j)={\rm Trace}_{K/\mathbb Q}(\zeta^{i-j}).
\end{equation}

If $i=j$, then \eqref{trac} and \eqref{trn} yield
${\rm tr}_{K, \theta}(\zeta^i, \zeta^j)=n$.

Let $i\neq j$. The element $\zeta^{i-j}$ is a primitive $2^s$th root of $1$
for some positive integer $s\leqslant a$; hence  its minimal
polynomial over $\mathbb Q$ is
$t^{2^{s-1}}+1$ (see \cite[Chap. VIII, \S3]{Lang}).
Since $0\leqslant i, j\leqslant 2^{a-1}-1$, we have $s\geqslant 2$.
Therefore, the degree of this polynomial is at least $2$.
This implies that ${\rm Trace}_{\mathbb Q(\zeta^{i-j})/\mathbb Q}(\zeta^{i-j})=0$.
In turn, from this, \eqref{trac}, and the equality
\begin{equation} \label{traces}
{\rm Trace}_{K/\mathbb Q}(x)=[K:\mathbb Q(x)]{\rm Trace}_{\mathbb Q(x)/\mathbb Q}(x)\quad\mbox{for every $x\in K$}
\end{equation}
(see, e.g., \cite[Chap.\;II, \S10, (9)]{ZS}),
we obtain that
${\rm tr}_{K, \theta}(\zeta^i, \zeta^j)=0$.

\eqref{2aciv} follows from \eqref{2aciii}, \eqref{m},
and Lemma \ref{prp}$\rm (\ref{m_d}_1)$.
\end{proof}

\begin{lemma}[{{\rm cf.\;\cite[Prop.\;9.1(i)]{BF2}}}]
\label{3bcycl} Fix a positive integer $b$.
Let $K$ be a $3^b$th cyclotomic field,
let $\theta$ be the complex conjugation, and let $\zeta\in K$
be a $3^b$th primitive root of unity. Then the following hold:
\begin{enumerate}[\hskip 4.2mm\rm(i)]
\item\label{3bci} $n=2\cdot 3^{b-1}$;
\item\label{3bcii} $\{\zeta^j\mid 0\leqslant j\leqslant 2\cdot 3^{b-1}-1\}$
is a basis of the $\Z$-module $\mathscr O$;
\item\label{3bciii} for every two elements $\zeta^i$ and $\zeta^j$ of this basis,
$${\rm tr}_{K, \theta}(\zeta^i, \zeta^j)=\begin{cases}n\quad&\mbox{if $i=j$},\\
0\quad &\mbox{if
$3^{b-1}$ does not divide $i-j,$}\\
-n/2\quad&\mbox{if
$i\neq j$ and
$3^{b-1}$ divides  $i- j;$}
\end{cases}$$
\item\label{3bciv} $m=n/2$.
\end{enumerate}
\end{lemma}
\begin{proof}
The same argument as in the proof of Lemma \ref{2acycl} yields
\eqref{3bci}, \eqref{3bcii}, and \eqref{3bciv}.

\eqref{3bciii} Formula \eqref{trac} still holds and the same argument
as in the proof of Lemma \ref{2acycl}\eqref{2aciii}
yields the validity of \eqref{3bciii}  for $i=j$. Assume
now that $i\neq j$. Then $\zeta^{i-j}$ is a primitive $3^s$th
root of unity for some positive integer
$s\leqslant b$. The degree of its minimal polynomial over $\Q$ is
\begin{equation}\label{deggam}
[\Q(\zeta^{i-j}):\Q]=\varphi(3^s)=2\cdot 3^{s-1}.
\end{equation}
Since the 3rd  cyclotomic polynomial is $t^2+t+1$, this minimal polyno\-mial is
$t^{2\cdot 3^{s-1}}+t^{3^{s-1}}+1$ (see \cite[Chap. VIII, \S3]{Lang}).
Given that $2\cdot 3^{s-1}-3^{s-1}=1$ if and only if $s=1$,
from this we infer that
\begin{equation}\label{trgam}
{\rm Trace}_{\Q(\zeta^{i-j})/\Q}(\zeta^{i-j})=
\begin{cases}0&\quad\mbox{if $s\geqslant 2$},\\
-1&\quad\mbox{if $s=1$}.
\end{cases}
\end{equation}

From  \eqref{3bci} and \eqref{deggam} we obtain $[K:\Q(\zeta^{i-j})]=[K:\Q]/[\Q(\zeta^{i-j}):\Q]=
n/(2\cdot 3^{s-1})=3^{b-s}$. This,
\eqref{traces}, and \eqref{trgam} then yield
\begin{equation}\label{trgam2}
{\rm Trace}_{K/\Q}(\zeta^{i-j})=
\begin{cases}0&\quad\mbox{if $s\geqslant 2$},\\
-n/2&\quad\mbox{if $s=1$}.
\end{cases}
\end{equation}
Clearly, \eqref{trgam2} is equivalent to \eqref{3bciii} for $i\neq j$.
\end{proof}

Given a $q$th cyclotomic field $K$ and a positive
integer $r$ dividing $q$, we denote by $K_r$ and
$\mathscr O_r$ respectively the unique $r$th
cyclotomic subfield of $K$ and its ring of integers.
They are ${\rm Aut}\,K$-stable; for $\alpha\in {\rm Aut}\,K$,
we denote the restriction $\alpha|_{K_r}$  still by $\alpha$.
If $\alpha$ is the complex conjugation of $K$, this
restriction is the complex conjugation of $K_r$.

\begin{lemma}\label{tenpro}
Let  $i$ and $j$ be coprime positive integers. Let $K$ be
a $ij$-th cyclotomic field and let $\theta$ be the complex conjugation.
Then the following hold:

\begin{enumerate}[\hskip 4.2mm\rm(i)]
\item\label{tenproi} the natural $\Q$-algebra homomorphism
\begin{equation}\label{cyte}
K_i\otimes_\Q K_j\to K,\quad x\otimes y\mapsto xy,
\end{equation}
is an isomorphism;
\item\label{tenproii}  for every $x\in K_i$ and $y\in K_j$,
the following equality holds:
\begin{equation}\label{trtrtr}
{\rm Trace}_{K/\Q}(xy)={\rm Trace}_{K_i/\Q}(x)\,{\rm Trace}_{K_j/\Q}(y).
\end{equation}
\item\label{tenproiii} the restriction  of  homomorphism
\eqref{cyte} to $\mathscr O_i\otimes_\mathbb Z\mathscr O_j$
is a ring isomorphism with
$\mathscr O$, which is also a lattice isometry
\begin{equation*}
(\mathscr O_i, {\rm tr}_{K_i, \theta})\otimes_\mathbb Z (\mathscr O_j, {\rm tr}_{K_j, \theta})\to
(\mathscr O, {\rm tr}_{K, \theta}).
\end{equation*}
\end{enumerate}
\end{lemma}
\begin{proof} Let $\zeta_i\in K_i$ and $\zeta_j\in K_j$
be respectively a primitive $i$th and a primitive $j$th root of unity.
Then $\zeta_i\zeta_j$ is a primitive $i j$th root of unity,
hence $K=\Q(\zeta_i\zeta_j)$. This shows that $K$ is
the compositum of $K_i$ and $K_j$; whence \eqref{cyte} is a surjective
homomorphism.
The coprimeness of $i$ and $j$ implies that $[K:\Q]=\varphi(ij)=\varphi(i)\varphi(j)=[K_i:\Q][K_j:\Q]$;
hence
\eqref{cyte} is an injective homomorphism. This proves \eqref{tenproi}.

By
definition of the trace of an element of a number field,
the left-hand side of \eqref{trtrtr} is the trace of the
$\Q$-linear transformation of $K$ given by multiplication
by $xy$. In view of \eqref{tenproi},  it is equal to the trace of the
$\Q$-linear transformation of $\Q(\zeta_i)\otimes_\Q\Q(\zeta_j)$
given by multiplication by $x\otimes y$. Hence
(see \cite[Chap.\,VII, \S5, no. 6]{Bo2}) it is equal to
the product of traces of the $\Q$-linear transformations of
$\Q(\zeta_i)$ and $\Q(\zeta_j)$ given by mul\-ti\-plications
by res\-pec\-tively $x$ and $y$. By the mentioned definition,
the latter product
is equal to the right-hand side of  \eqref{trtrtr}.
This proves\;\eqref{tenproii}.

In view of the equality $\mathscr O=\mathscr O_i\mathscr O_j$
(see \cite[Chap. IV, \S1, Thm. 4]{Lang2}, \cite[Thm. 2.6]{Wa})
the first statement in \eqref{tenproiii} follows
from \eqref{tenproi}. The second
follows from the first in view of \eqref{ti},
\eqref{trtrtr}, and the definition of tensor product of
lattices (see  \cite[Chap.\;1, \S5]{MH}).
\end{proof}

In the next lemma, we use  that if $K$ is totally real,
then ${\rm discr}\,{K/\mathbb Q}$ is positive
(see, e.g., \cite[Prop.\;1.2a]{Koch}).

\begin{lemma}
\label{lowerDreal}
Let $K$ be a totally real number field of degree $n\!>\!1$ over\;$\mathbb Q$.
\begin{enumerate}[\hskip 2.2mm\rm(i)]
\item\label{52}
If $n\leqslant  24$, then $\sqrt[n]{{\rm discr}\,{K/\mathbb Q}}>n$.
\item\label{54}
If $3 \leqslant  n \leqslant  75$, then
$\sqrt[n]{{\rm discr}\,{K/\mathbb Q}}>(n+1)^2/2n$.
\end{enumerate}
\end{lemma}

\begin{proof}
This readily follows from
 \cite[Table 2 (Cas To\-ta\-lement Reel), p.\,1]{Diaz}.
\end{proof}

\begin{lemma}
\label{lowerDCM}
Let $K$ be a CM-field of degree $n>2$ over $\mathbb Q$.
If $n\leqslant  48$, then $\sqrt[n]{|{\rm discr}\,{K/\mathbb Q}|}>n/2$.
\end{lemma}

\begin{proof} In view of \eqref{pCM}, we have the inequality
\begin{equation}\label{KK}
\textstyle\sqrt[n]{|{\rm discr}\,{K/\mathbb Q}|}\geqslant \!\!\sqrt[n/2]{{\rm discr}\,{K^\theta/\mathbb Q}}.
\end{equation}
The claim then readily follows from
the lower bounds on the right-hand side of \eqref{KK}, which is
obtained by applying
Lem\-ma\;\ref{lowerDreal}\eqref{52}
to
the totally real number field $K^\theta$  of degree $n/2$.
\end{proof}

\begin{proposition}\label{comp}
Let
$E_K$ be the $n$-element set of all field
embed\-dings $K \hookrightarrow \C$
and let $b_K$ be
the $\mathbb Q$-bilinear map
\begin{equation}\label{cem}
b_K\colon K\times K\to \mathbb C, \quad b_K(x, y):=\textstyle\sum_{\sigma\in E_K}\sigma(x)\overline{\sigma(y)}.
\end{equation}
Then the following hold:
\begin{enumerate}[\hskip 2.2mm\rm(i)]
\item The $\mathbb Q$-linear span of $b_K(K\times K)$ is a proper subset of $\mathbb R$
    con\-taining
    $\mathbb Q$.
\item $b_K$ is symmetric and positive-definite.
\item
Properties {\rm (a), (b), (c)} listed below are equivalent:
\begin{enumerate}[\hskip 0mm \rm(a)]
\item $b_K(K\times K)=\mathbb Q$.
\item $b_K({\mathscr O}\times {\mathscr O})\subseteq\mathbb Q$.
\item There is an involutive field automorphism $\tau\in {\rm Aut}\,K$ such that
    $b_K={\rm tr}_{K, \tau}$.
\end{enumerate}
\item If {\rm (c)} holds, then either $K$ is totally real and $\tau={\rm id}$, or $K$ is a CM-field and $\tau$ is the complex conjugation.
\end{enumerate}
\end{proposition}

\begin{proof}
Below we use the notation
\begin{equation*}
\iota\colon \mathbb C\to\mathbb C, \quad \iota(z):=\overline{z}.
\end{equation*}

(i) In view of \eqref{cem} and \eqref{sp}, for every  $q\in \mathbb Q$, we have $b_K(q, 1)=nq$; whence the inclusion $\mathbb Q\subseteq b_K(K\times K)$.

For every $\sigma\in E_K$, we have $\iota\circ \sigma\in E_K$; therefore,
\begin{equation}\label{perm}
\iota\circ E_K:=\{\iota\circ \sigma\mid \sigma\in E_K\}=E_K.
\end{equation}
From \eqref{cem} and \eqref{perm}, for every $x, y\in K$, we deduce the following:
\begin{equation}
\label{R}
\begin{split}
\overline{b_K(x, y)}&=\textstyle\sum_{\sigma\in E_K}\overline{\sigma(x)}\sigma(y)=
b_K(y, x)\\
&=
\textstyle\sum_{\sigma\in E_K}(\iota\circ\sigma)(x) \overline{(\iota\circ\sigma)(y)}
=\textstyle\sum_{\delta\in \iota\circ E_K}\delta(x)\overline{\delta(y)}\\
&=
\textstyle\sum_{\delta\in E_K}\delta(x)\overline{\delta(y)}=b_K(x, y);
\end{split}
\end{equation}
whence the inclusion $b_K(K\times K)\subseteq \mathbb R$. Since $K$ is countable, while $\mathbb R$ is not, this inclusion is proper.

(ii) follows from \eqref{R} and \eqref{cem}.

(iv)
follows from (ii) and
 Lemma \ref{TRCM}.

(iii) Since $\mathscr O\subset K$, we have (a)$\Rightarrow$(b), and since $K$ is the $\mathbb Q$-linear span of $\mathscr O$, we have (b)$\Rightarrow$(a).

Next, (c)$\Rightarrow$(a) because of \eqref{ti}.

Conversely, assume that (a) holds. Then we have two nondegenerate $\mathbb Q$-bilinear maps
\vskip -7mm
\begin{equation*}
\xymatrix@C=3mm{K\times K\hskip 5mm\ar@/_1.0pc/[rr]_{\hskip 3mm b_K}\ar@/^1.0pc/[rr]^{{\rm tr}_{K, {\rm id}}}&&\mathbb Q.}
\end{equation*}

Therefore, by \cite[Chap. XIII, \S5]{Lang} and \eqref{ti}, there is a nondegenerate $\mathbb Q$-linear map $\tau\colon K\to K$ such that
\begin{equation}\label{tau}
b_K(x, y)={\rm Trace}_{K/\mathbb Q}^{\ }(x\!\cdot\! \tau(y))\quad \mbox{for all $x, y\in K$}.
\end{equation}
We claim that $\tau$ is an involutive field automorphism of $K$; if this is proved, then
\eqref{tau}, \eqref{ti} show that (c) holds.

To prove the claim, we notice that using
\eqref{cem} and \eqref{sum}, we can rewrite \eqref{tau} as follows:
\begin{equation}\label{rew}
\textstyle\sum_{\sigma\in E_K}\sigma(x)\overline{\sigma(y)}=
\sum_{\sigma\in E_K} \sigma(x)\sigma(\tau(y))
\quad \mbox{for all $x, y\in K$}.
\end{equation}
Since for every $\sigma\in E_K$, the map $K^\times \to \mathbb C^\times$, $x\mapsto \sigma(x)$,  is a group homomorphism,
by Artin's theorem on the linear independence of cha\-racters
\cite[Chap. VIII, \S4, Thm. 7]{Lang},  we conclude from \eqref{rew} that
\begin{equation}\label{ati}
\overline{\sigma(y)}=\sigma(\tau(y))\quad\mbox{for all $\sigma\in E_K$, $y\in K$}.
\end{equation}
In turn, \eqref{ati} implies that the subfield $\sigma(K)$ of $\mathbb C$ is $\iota$-stable and
the following diagram, in which $\gamma:=\iota|_{\sigma(K)}$,
is commutative:
\begin{equation}\label{diagr}
\xymatrix@C=7mm{
K\ar[r]^{\hskip -2mm\sigma}\ar@/_1.6pc/[rrr]_{\hskip 3mm \hskip -3mm\tau}&\sigma(K)\ar[r]^{\gamma}&\sigma(K)\ar[r]^{\hskip 3mm\sigma^{-1}}&K.
}
\end{equation}
Since each of the upper arrows in \eqref{diagr} is a field isomorphism, $\tau$ is a field automorphism. Moreover, since $\tau=\sigma^{-1}\circ\gamma\circ\sigma$ and $\gamma$ (being the restriction of involution $\iota$) is involutive, $\tau$ is involutive as well. This completes the proof.
\end{proof}

\section{When is $(\mathscr O, {\rm tr}_{K, \theta})$
a root lattice?}\label{rola}
\begin{proof}[Proof of Theorem {\rm \ref{ridnew}}]
In view of Examples \ref{ex1}, \ref{ex2}, \ref{ex3},
the ``if'' part is clear. To prove the ``only if'' one, suppose that
$({\mathscr O}  , {\rm tr}_{K, \theta})$ is a root lattice.
Then by Corollary \ref{fie},  either $K$ is a totally real
field and $\theta={\rm id}$, or $K$ is a CM-field
and $\theta$ is the complex conjugation.

There are two possibilities, which we will consider separately:
\begin{enumerate}[\hskip 4.2mm\rm(a)]
\item the lattice $({\mathscr O}  , {\rm tr}_{K, \theta})$
is primitive or, equivalently,
at least one of its indecomposable components is not isometric to
$\mathbb A_1$.

\item the lattice $({\mathscr O}  , {\rm tr}_{K, \theta})$
is nonprimitive or, equivalently, it is iso\-met\-ric to
$\mathbb A_1^{a_1}$ for some $a_1$.
\end{enumerate}
Our considerations below exploit the fact that
\begin{equation}\label{dA}
{\rm discr}\hskip .4mm{\mathbb A}_\ell=\ell+1
\quad \text{for every $\ell$,}
\end{equation}

We first assume that (a) holds. Then
\begin{equation}\label{Kt1new}
m^{\ }=1.
\end{equation}

If there is an indecomposable component of
$({\mathscr O}, {\rm tr}_{K, \theta})$ isometric to
$\Z^1$,
then $(\mathscr O, {\rm tr}_{K, \theta})[1]\neq\varnothing$.
By Lemma \ref{tnin}\eqref{btnin},
this and \eqref{Kt1new} imply
that
$n=1$, i.e.,  $K=\mathbb Q$ and $\theta={\rm id}$.

If every indecomposable component of
$({\mathscr O}, {\rm tr}_{K, \theta})$ is not isometric to $\Z^1$,
then
$(\mathscr O, {\rm tr}_{K, \theta})$ is an even
lattice and
$(\mathscr O, {\rm tr}_{K, \theta})[2]\nobreak
\neq\varnothing$. By Lemma \ref{tnin}\eqref{ctnin},
from this and \eqref{Kt1new} we deduce that $n=2$, i.e.,
\begin{equation}\label{fienew}
K=\mathbb Q\big(\sqrt{c}\big) \quad \text{for a square free integer $c$}.
\end{equation}
Since the sum of ranks of all indecomposable
components of $(\mathscr O, {\rm tr}_{K, \theta})$ is $n$,
we infer from $n=2$ and
(a)
that $(\mathscr O, {\rm tr}_{K, \theta})$ is
isometric to $\mathbb A_2$.
  This, \eqref{Kt1new},  \eqref{dA}, and
  Lem\-ma \ref{prp}\eqref{test}, $\rm (\ref{m_d}_1)$, \eqref{dgdg}
 then imply that
\begin{equation}\label{Da2}
|{\rm discr}\,K/\mathbb Q|=3.
\end{equation}
On the other hand, \eqref{fienew} yields
\begin{equation}\label{ffdis}
{\rm discr}\,K/\mathbb Q=\begin{cases}
4c& \text{ if } c\equiv 2, 3 \;{\rm mod}\, 4,\\
c  & \text{ if } c\equiv 1 \; {\rm mod}\, 4
\end{cases}
\end{equation}
(see, e.g., \cite[Chap.\;13, Prop.\;13.1.2]{IR}).
From
\eqref{ffdis} and \eqref{Da2}
we conclude  that $c=-3$; whence
$K=\mathbb Q(\sqrt{-3})$ and $\theta$ is
the complex conjugation. This completes the
consideration of case (a).

Now we assume that (b) holds. Then $m=2$ and
$({\mathscr O},
{{\rm tr}_{K, \theta}/m})$ is isometric to $\mathbb Z^{a_1}$.
Hence $({\mathscr O}, {{\rm tr}_{K, \theta}/m})[1]\neq\varnothing$.
By Lemma \ref{tnin}\eqref{btnin}, this yields $n=a_1=2$,
hence again \eqref{fienew} and \eqref{ffdis} hold.
From \eqref{dA} and Lemma \ref{prp}\eqref{dgdg}
we obtain $|{\rm discr}\,K/\mathbb Q|=4$.
In view of \eqref{ffdis}, this yields $c=-1$; whence
$K=\mathbb Q(\sqrt{-1})$ and $\theta$ is the
complex conjugation.
This completes the consi\-de\-ra\-tion of case  (b).
\end{proof}

\section{When is $(\mathscr O, {\rm tr}_{K, \theta})$
similar to $\mathbb Z^n$?}\label{szn}

\begin{proof}[Proof of Theorem {\rm \ref{pureodd}}]
Since $\mathbb Z^n$ is primitive, we have
\eqref{pureodda}$\Leftrightarrow$\eqref{pureoddc}.
If a lattice $(L, b)$ is isometric to $\mathbb Z^1$,
then $(L, 2b)$ is isometric to $\mathbb A^1$. This yields
implications
\eqref{pureodda}$\Leftrightarrow$\eqref{pureoddb} and \eqref{pureoddc}$\Leftrightarrow$\eqref{pureoddd}.

\vskip 1mm

\eqref{pureodda}$\Rightarrow$\eqref{pureodde}
Suppose that  $(\mathscr O, {\rm tr}_{K, \theta})$
is similar to $\mathbb Z^n$.
Then, by Corollary \ref{fie}, either $K$ is a totally
real field and $\theta={\rm id}$, or $K$ is a CM-field
and $\theta$ is the complex conjugation. Since
$(\mathscr O, {\rm tr}_{K, \theta}/m)$ and $\mathbb Z^n$
are isometric, and the $\Z$-module $\mathbb Z^n$ is
generated by
$\mathbb Z^n[1]$, we deduce from Lemma \ref{tnin}\eqref{btnin} that
\begin{equation}\label{m=n}
m=n,
\end{equation}
  and \eqref{m=n} implies by Lemma \ref{tnin}\eqref{3} that
  \begin{equation}\label{span}
  \mbox{$\mathscr O$ is the $\mathbb Z$-linear span of $\mu_K$.}
  \end{equation}

Suppose that the order of the cyclic group $\mu_K$
is divisible by an odd prime $p$. Then there is
$x\in \mu_K$
which is a $p$th primitive root of unity.
The minimal polynomial of $x$ over
$\mathbb Q$ is
$t^{p-1}+t^{p-2}+\cdots+1$, hence
 \begin{equation}\label{trdim}
 {\rm Trace}_{\Q(x)/\Q}(x)=-1\quad \mbox{and}
 \quad [\Q(x):\Q]=p-1\ge 2.
  \end{equation}
  Since ${\rm Trace}_{K/\mathbb Q}(x)=
  [K:\mathbb Q(x)]{\rm Trace}_{\Q(x)/\Q}(x)$, we
  obtain from \eqref{trdim} that
\begin{equation}\label{fractio}
0<{\rm Trace}_{K/\Q}(-x)=\frac{n}{p-1}<n.
\end{equation}
However, in view of \eqref{mmmm}
and \eqref{m=n}, the integer ${\rm Trace}_{K/\Q}(-x)$
is divisible by $n$. This contradics \eqref{fractio}.
The obtained contradiction proves that there a positive
integer $a$ such that $\mu_K$ is a cyclic group of
order $2^a$. In view of \eqref{span}, this implies that $K$
is a $2^a$th cyclotomic field. In particular, it is a CM-field,
and therefore, $\theta$ is the complex conjugation.
This proves implication \eqref{pureodda}$\Rightarrow$\eqref{pureodde}.

\vskip 1mm

\eqref{pureodde}$\Rightarrow$\eqref{pureoddc}
Let $K$ be a $2^a$th cyclotomic field for
a positive integer $a$
and let $\theta$ be the complex conjugation.
By Lemma \ref{2acycl} (whose notation we use),
$n=2^{a-1}$, $m=n$, and
$(\mathscr O, {\rm tr}_{K, \theta}/m)$ is the orthogonal
direct sum of the sublattices
$\mathbb Z\zeta^j$, $0\leqslant j\leqslant 2^{a-1}-1$,
each of which is isometric to\;$\mathbb Z^1$.
This proves implication
\eqref{pureodde}$\Rightarrow$\eqref{pureoddc}.

The last statement of Theorem \ref{pureodd}
follows from Lemma \ref{2acycl}\eqref{2aciii}.
\end{proof}

\section{When is $(\mathscr O, {\rm tr}_{K, \theta})$
similar to\\ an even primitive root lattice?}\label{prprl}

\begin{proof}[Proof of Theorem {\rm \ref{simeven}}] The implications
\eqref{simeva}$\Leftrightarrow$\eqref{simevb}
and \eqref{simevc}$\Rightarrow$\eqref{simevb} are clear.

\eqref{simeva}$\Rightarrow$\eqref{simevd}
Suppose that  $(\mathscr O, {\rm tr}_{K, \theta})$
is similar to an even primitive root lattice $(L, b)$.
Then, by Corollary \ref{fie}, either $K$ is a totally
real field and $\theta={\rm id}$, or $K$ is a CM-field
and $\theta$ is the complex conjugation. Since
$(\mathscr O, {\rm tr}_{K, \theta}/m)$ and $(L, b)$
are isometric, and $(L, b)$ is an even lattice
generated by $(L, b)[2]$, we deduce from
Lemma \ref{tnin}\eqref{ctnin} that
\begin{equation}\label{m=n/2}
m=n/2,
\end{equation}\vskip 2mm\noindent
  and \eqref{m=n/2} implies by Lemma \ref{tnin}\eqref{3}
  that $\mathscr O$ is the $\mathbb Z$-linear span of
  $\mu_K$.
  Hence $K$ is a cy\-c\-lo\-tomic field and $\theta$
  is the complex conjugation.

There is an odd prime $p$ dividing  the order of
$\mu_K$: otherwise, this order is $2^a$ for some $a$,
hence  $m=n$ by Theorem \ref{pureodd}, which
contradicts \eqref{m=n/2}. Thus, we can find an element
$x\in \mu_K$ which is a $p$th primitive root of
unity. The same argument as in the above proof of implication \eqref{pureodda}$\Rightarrow$\eqref{pureodde} of
Theorem \ref{pureodd}
shows that
\eqref{fractio} holds. In view of
\eqref{mmmm} and \eqref{m=n/2}, the
integer ${\rm Trace}_{K/\Q}(-x)$ is divisible by $n/2$.
This and \eqref{fractio} then yield $p-1=2$, i.e., $p=3$.
Thereby we proved that $K$ is a $2^a3^b$th cyclotomic
field for some $a\geqslant 0$, $b>0$.
Since every $3^b$th cyclotomic field is simul\-taneously
a $2^13^b$th cyclotomic field, we may assume that $a>0$.
This com\-p\-letes the proof of
implication
\eqref{simeva}$\Rightarrow$\eqref{simevd}.

\eqref{simevd}$\Rightarrow$\eqref{simevc}
Let $K$ be a $2^a3^b$th cyclotomic field for some $a>0$, $b>0$.
We use the notation introduced in the paragraph
immediately preceding Lemma \ref{tenpro} and put
\begin{equation}\label{notatio2}
i:=2^a, \;\;j:=3^b,\;\; n_i:=[K_i:\Q],\;\; n_j:=[K_j:\Q].
\end{equation}\vskip 2mm\noindent
By Lemma \ref{tenpro}\eqref{tenproi} we have $n=n_in_j$.
Hence by Lemma \ref{tenpro}\eqref{tenproiii} there
is a lattice isometry
\begin{equation}\label{isotens}
(\mathscr O_i, {\rm tr}_{K_i, \theta}/n_i)\otimes_\mathbb Z \big(\mathscr O_j, {\rm tr}_{K_j, \theta}/(n_j/2)\big)\to
\big(\mathscr O, {\rm tr}_{K, \theta}/(n/2)\big).
\end{equation}\vskip 2mm

In view of \eqref{notatio2} and Lemmas \ref{2acycl}, \ref{3bcycl},
the lattices
$(\mathscr O_i, {\rm tr}_{K_i, \theta}/n_i)$  and
$\big(\mathscr O_j, {\rm tr}_{K_j, \theta}/(n_j/2)\big)$
are isometric to respectively
$\Z^{n_i}$ and ${\mathbb A}_2^{n_j/2}$.
Since
the lattices
$\mathbb Z^1\otimes_{\mathbb Z} \mathbb A_2$ and
$\mathbb A_2$ are isometric, the existence of
isomorphism \eqref{isotens} then implies that the lattice
$\big(\mathscr O, {\rm tr}_{K, \theta}/(n/2)\big)$ is isometric to
$\mathbb A_2^{n/2}$.
This com\-p\-letes the proof of
implication
\eqref{simevd}$\Rightarrow$\eqref{simevc}.

The last two statements of Theorem \ref{simeven}
follow from Lemma \ref{2acycl}\eqref{2aciii},\eqref{2aciv}, Lemma \ref{3bcycl}\eqref{3bciii},\eqref{3bciv}, and
Lemma \ref{tenpro}\eqref{tenproii}.
\end{proof}

\section{When is $(\mathscr O, {\rm tr}_{K, \theta})$ similar to\\
a positive-definite even unimodular lattice
of rank $\leqslant 48$?}\label{prpul}

\begin{proof}[Proof of Theorem {\rm \ref{24}}]

Arguing on the contrary, assume that $(\mathscr O, {\rm tr}_{K, \theta})$
is similar to a positive-definite even unimodular lattice
$(L, b)$ of rank $n$, where
\begin{equation}\label{40}
8\leqslant n\leqslant 48.
 \end{equation}
(the first inequality follows from the fact that $8$ divides
$n$, see, e.g., \cite[Chap. 7, \S6, Cor.\;18]{CS}).
Since
$(L, b)$ is unimodular, we have
\begin{equation}\label{di1}
|{\rm disc}\,(L, b)|=1,
\end{equation}
and  \eqref{di1} implies, by Lemma \ref{prp}\eqref{test},
that $(L, b)$ is primitive.  Therefore,
$(L, b)$ is isometric to $(\mathscr O, {\rm tr}_{K, \theta})$.
By  Lemma \ref{TRCM}, since $(L, b)$ is positive-definite,
the pro\-per\-ties listed in item
\eqref{co} of this lemma hold. In view of Lemma
\ref{prp}\eqref{test},\eqref{dgdg} and \eqref{di1}, we have
\begin{equation}\label{pdel}
\sqrt[n]{|{\rm disc}\,K/\mathbb Q|}=m.
\end{equation}
 By Lemma \ref{prp}\eqref{te}, since $(L, b)$ is even,
 $n/2\equiv 0\;{\rm mod}\,m$.
 This and \eqref{pdel} yield
 \begin{equation}\label{pdel2}
\sqrt[n]{|{\rm disc}\,K/\mathbb Q|}\leqslant n/2.
\end{equation}
On the other hand, by Lemma \ref{lowerDreal}\eqref{54},
Lemma \ref{lowerDCM},
and
\eqref{40}, for $n\leqslant 48$, we have the inequality
\begin{equation*}\label{pde3}
\sqrt[n]{|{\rm disc}\,K/\mathbb Q|}\!>\!
\begin{cases}(n+1)^2/2n\!>\!n/2&\mbox{if $K$ is a totally  real field},\\
n/2&\mbox{if $K$ is a CM-field},
\end{cases}
\end{equation*}
which contradicts \eqref{pdel2}.
\end{proof}

\section{Appendix 1: The $\mathbb Z^n$, $\mathbb A_n$,
$\mathbb D_n$,
and $\mathbb E_n$ lattices}\label{ap}

To be self-contained, here we briefly describe the lattices
$\mathbb Z^n$, $\mathbb A_n$, $\mathbb D_n$, and
$\mathbb E_n$ as they play a key role in this paper.
For details and a discussion of their properties see
\cite{Bo}, \cite{CS}, \cite{Ma}.

Let $\mathbb R^m$ be the $m$-dimensional
coordinate real vector space of rows
endowed with the standard Euclidean structure
\begin{equation}\label{multip}
\mathbb R^m\times \mathbb R^m\to \mathbb R,\quad
((x_1,\ldots, x_m), (y_1,\ldots, y_m)):=\textstyle \sum_{j=1}^m x_jy_j.
\end{equation}

Denote by $e_j$ the row $(0,\ldots, 0,1,0,\ldots, 0)$,
where the number of $0$'s to the left of $1$ is $j-1$.

If $L$ is the $\mathbb Z$-linear span of a set of linearly
independent elements of $\mathbb R^m$ such that
$b(L\times L)\subseteq \mathbb Z$, where $b$ is the
restriction of map \eqref{multip} to $L\times L$,
then $(L, b)$ is called a {\it lattice in} $\mathbb R^m$
and denoted just by  $L$.

With these notation and conventions, we have:

${\mathbb Z}^n$ is the lattice
$\{(x_1,\ldots, x_n)\mid  x_j\in \mathbb Z \;\;\mbox{for all $j$}\}$
in  $\mathbb R^n$.

 ${\mathbb A}_n$ is the lattice
$\{(x_1,\ldots, x_{n+1})\in \mathbb Z^{n+1}\mid
\textstyle\sum_{j=1}^{n+1}x_j
=0\}$
in $\mathbb R^{n+1}$.

${\mathbb D}_n$ is the lattice
$\{(x_1,\ldots, x_n)\in \mathbb Z^{n}\mid
\textstyle \sum_{j=1}^{n}x_j
\equiv\,0\;{\rm mod}\,2\}$
in $\mathbb R^{n}$,
$n\geqslant 4$.

 ${\mathbb E}_8$ is the lattice $\mathbb D_8\bigcup \big(\mathbb D_8+\tfrac{1}{2}(e_1+\cdots+e_8)\big)$ in $\mathbb R^{8}$.

 ${\mathbb E}_7$ is the orthogonal in $\mathbb E_8$ of
the sublattice $\mathbb Z(e_7+e_8)$.

${\mathbb E}_6$ is the orthogonal in $\mathbb E_8$ of
the sublattice $\mathbb Z(e_7+e_8)+\mathbb Z(e_6+e_8)$.

Each of these lattices except $\mathbb A_1$ is primitive. Each of them
except $\Z^n$
for every $n$
is even.

\section{Appendix 2: Cyclicity criterion}\label{ap2}

\begin{theorem}[{{\rm cyclicity criterion}}]\label{cdg}
Let $n>1$ and let $p$ be a prime  ra\-mi\-fied in $K$.
Then the fol\-lowing properties
of $\mathscr O$ are equivalent:
\begin{enumerate}[\hskip 2.2mm\rm(i)]
\item\label{i} The $p$-primary component of the
additive group of ring $\mathscr O/\mathfrak d$ is a  cyclic group
\textup(automatically nontrivial\textup).
\item\label{ii} The following conditions hold:
\begin{enumerate}[\hskip 2mm\rm(a)]
\item[$({\rm{ii}}_1)$]  in $\mathscr O$, there is exactly one
ramified prime ideal $\mathfrak p$ which
lies
over\;$p$;
\item[$({\rm {ii}}_2)$] $p$ is odd,
${\rm ord}_{\mathfrak p}\, p\mathscr O=2$,
and
$\mathscr O/\mathfrak p$ is the field of $p$ elements.
\end{enumerate}
\end{enumerate}
\end{theorem}
\begin{proof}
Let $\mathfrak r_1,\ldots, \mathfrak r_m$ be
all pairwise distinct
prime ideals of
$\mathscr O$
ramified in $K/\mathbb Q$. For every $\mathfrak r_i$,
there is a prime integer $r_i$ and a positive integer $f_i$ such that
\begin{equation}\label{rii}
\mathfrak r_i\cap \mathbb Z=r_i\mathbb Z\quad\mbox
{and}\quad |\mathscr O/{\mathfrak r}_i^{s}|=r_i^{f_is}\;\;
\mbox{for every positive integer $s$}
\end{equation}
(see, e.g., \cite[Chap.\;12, \S2, 3]{IR}). We put
\begin{equation}\label{nota}
e_i:={\rm ord}_{\mathfrak r_i}r_i\mathscr O.
\end{equation}
As $\mathfrak r_i$ is ramified, $e_i\geqslant 2$.
There are positive integers $d_1,\ldots, d_m$ such that
\begin{equation}\label{decdif}
\mathfrak d=\mathfrak r_1^{d_1}\cdots \mathfrak r_m^{d_m}
\end{equation}
(see, e.g., \cite[Chap.\;III, Thm.\;(2.6)]{Neukirch}).
By Dedekind's theorem (see, e.g., \cite[Thm. (2.6), p.\,199]{Neukirch}),
we have
\begin{equation}\label{De}
d_i
\geqslant e_i-1,\;
\mbox{where the equality holds if and only if $r_i\nmid e_i$}.
\end{equation}

By the Chinese remainder theorem (see, e.g., \cite[Props. 12.3.1]{IR}),
decomposition \eqref{decdif} yields
the following ring isomorphism:
\begin{equation}\label{ri}
\mathscr O/\mathfrak d\approx \mathscr O/{\mathfrak r}_1^{d_1}\oplus\cdots \oplus
\mathscr O/{\mathfrak r}_m^{d_m}.
\end{equation}

As every $r_i$ is prime, \eqref{ri},
\eqref{rii} imply that (i) is satisfied if and only if
the following two conditions (a) and (b) hold:

(a) there is exactly one $i$ such that $r_i=p$,

(b) the additive group of $\mathscr O/\mathfrak r_i^{d_i}$ is cyclic.

Clearly (a) is equivalent to $({\rm{ii}}_1)$.\;We
shall now show that (b) is equivalent to
$({\rm{ii}}_2)$ with $\mathfrak p=\mathfrak r_i$
and $p=r_i$; this will complete the proof.

Assume that (b) holds.
By \eqref{nota}, we have  ${\mathfrak r}_i^{e_i-1}
\supseteq {\mathfrak r}_i^{e_i}\supseteq r_i\mathscr O$. Hence
there are the ring epimorphisms
\begin{equation}\label{epi}
\mathscr O/r_i\mathscr O\twoheadrightarrow \mathscr O/{\mathfrak r}_i^{e_i}\twoheadrightarrow \mathscr O/{\mathfrak r}_i^{e_i-1}.
\end{equation}
Since $r_i$ is prime, it is the order of every nonzero element of
the additive group of $\mathscr O/r_i\mathscr O$.
The existence of
epimorphisms \eqref{epi} then implies that
\begin{equation}\label{ordd}
\begin{gathered}
r_i\;\; \mbox{is the order of every nonzero element of}\\[-1.5mm]
\mbox{the additive groups of $\mathscr O/{\mathfrak r}_i^{e_i}$
and $\mathscr O/{\mathfrak r}_i^{e_i-1}$}.
\end{gathered}
\end{equation}

We consider two possibilities stemming from
\eqref{De}:
\begin{enumerate}[\hskip 4.2mm \rm (i)]
\item[(T)] $d_i=e_i-1$ (tamely ramified ${\mathfrak r}_i$);
\item[(W)] $d_i\geqslant e_i$ (wildly ramified ${\mathfrak r}_i$).
\end{enumerate}

First, assume that (T) holds. Then from (b), \eqref{ordd}, and
\eqref{rii} we infer that
$f_i=d_i=1$, $|\mathscr O/\mathfrak r_i|=r_i$. Hence $e_i=2$, and, by \eqref{De}, $r_i$ is odd.
Thus, as claimed, in this case, $({\rm {ii}}_2)$ is fulfilled.

Next, assume that (W) holds. Then ${\mathfrak r}_i^{e_i}
\supseteq {\mathfrak r}_i^{d_i}$; whence there is a ring
epimorphism
\begin{equation}
\label{ippi}
{\mathscr O}/{\mathfrak r}_i^{d_i}
\twoheadrightarrow \mathscr O/{\mathfrak r}_i^{e_i}.
\end{equation}
From (b), \eqref{ordd}, \eqref{ippi}, and \eqref{rii} we
then infer that $r_i=|{\mathscr O}/{\mathfrak r}_i^{e_i}|=r_i^{f_ie_i}$
contrary to the inequality $e_i\geqslant 2$. Thus, (W) is impossible.

This completes the proof of (b)$\Rightarrow$$({\rm{ii}}_2)$
with $\mathfrak p=\mathfrak r_i$ and $p=r_i$.
The converse implication is immediate.
\end{proof}

\end{document}